\documentclass[12pt,reqno]{amsart}
\usepackage{amssymb,a4wide,eucal,amscd,yhmath,amsgen,amsopn}
\usepackage{fancyhdr}
\usepackage{amsmath,hyperref,xcolor}
\usepackage{mathrsfs}
\usepackage[all,cmtip]{xy}
\SilentMatrices
\SelectTips{eu}{}
\newtheorem{theorem}{Theorem}[section]
\newtheorem{proposition}[theorem]{Proposition}

\newtheorem{lemma}[theorem]{Lemma}
\newtheorem{Not}[theorem]{Notation}
\newtheorem{rem}[theorem]{Remark}
\theoremstyle{definition}
\newtheorem{definition}[theorem]{Definition}

\newcommand\projective\mathbf
\newcommand\PP{\projective P}

\newcommand\OO{\mathcal O}

\newcommand\onto\twoheadrightarrow
\newcommand\lra\longrightarrow
\newcommand\dar\downarrow

\DeclareMathOperator{\Hom}{Hom}

\begin{document}

\title{A Cohomological criterion for the splitting of vector bundles on $\PP^{n_1}\times\cdots\times\PP^{n_s}$}
\author{Damian M Maingi}
\date{November, 28th 2025}

\address{Department of Mathematics\\Catholic University of Eastern Africa\\P.O Box 30197, 00100 Nairobi, Kenya\\https://orcid.org/0000-0001-9267-9388}
\email{dmaingi@cuea.edu, dmaingi@gmail.com}

\begin{abstract}
In this paper we study the cohomological criterion for the splitting of vector bundles on multiprojective spaces
$\PP^{n_1}\times\ldots\times\PP^{n_s}$. We also give a generalization of vanishing cohomological criteria for vector bundles on 
$\PP^{n}\times\ldots\times\PP^{n}$.
\end{abstract}

\maketitle

\section{Introduction}

\noindent The goal of this paper is to study cohomological criterion for the splitting of vector bundles on multiprojective spaces
$\PP^{n_1}\times\ldots\times\PP^{n_s}$. The Horrocks theorem \cite{12} says that a vector bundle $E$ on a projective space $\mathbb{P}^n$ splits
into a direct sum of line bundles if and only if it has no intermediate cohomology that is $H^i(E(t))=0$ for all $1\leq{i}\leq{n-1}$ and $t\in\mathbb{Z}$.
This theorem gives a criterion for a vector bundle to be decomposable in to simpler components by checking the vanishing of its intermediate
cohomologies. We extend the results of Miyazaki in \cite{15} and \cite{16} where the author studied cohomological criterion for the splitting of vector bundles on 
the biprojective space $\PP^{n_1}\times\PP^{n_2}$ over an algebraically closed field $k$.
Ballico and Malaspina in \cite{4} studied regularity and splitting  conditions on multiprojective spaces. Maingi in \cite{14}
extended their results to $\PP^{n_1}\times\ldots\times\PP^{n_s}$.


\noindent The notion of regularity on a multiprojective spaces has been tackled by several authors for instance \cite{4,6,7,14,15,16,18}.
Ballico and Malaspina in \cite{4} gave modified notion of the Hoffman and Wang \cite{18} regularity on $\PP^n\times\PP^m$ and extended it to  multiprojective space 
$\PP^{n_1}\times\cdots\times\PP^{n_s}$. In this paper we extend the splitting criteria for vector bundles on a biprojective space $\PP^{n}\times\PP^{m}$  to a multiprojective space $\PP^{n_1}\times\cdots\times\PP^{n_s}$.

\begin{Not}
\noindent Our ambient space is $X=\PP^{n_1}\times\ldots\times\PP^{n_s}$.
For a coherent sheaf $E$ we denote by $E\otimes\OO_X(a_1, \cdots, a_s):= {p_1}^*\OO_{\PP^{n_1}}(a_1)\otimes\cdots\otimes{p_s}^*\OO_{\PP^{n_s}}(a_s)$
where $p_i$ are natural projections  from $X$ onto $\PP^{n_i}$, for all $i=1,\cdots,s$.
For the sake of brevity we shall use the notation $H^q(\mathscr{F})$ in place of $H^q(X,\mathscr{F})$.
\end{Not}

\noindent The main motivation for this work is derived from splitting criterion of vector bundles enuntiated by Horrocks in \cite{12}.
We extend the results of Miyazaki in \cite{14} from biprojective spaces $\PP^{n_1}\times\PP^{n_2}$ to multiprojective spaces
$\PP^{n_1}\times\cdots\times\PP^{n_s}$ in the following two theorems. The third theorem is a consequence of the second theorem.\\

\noindent The main results  of this paper are the following:

\begin{theorem}
 Let $E$ be a rank $r$ vector bundle on $X=\PP^{n_1}\times\cdots\times\PP^{n_s}$ with $n_1,\dots,n_s\geq2$.
Then $E$ is a direct sum of line bundles of
\begin{enumerate}
 \item $\OO_X$
  \item $\OO_X(0,\ldots,0,1)$, $\OO_X(0,\ldots,1,0),\ldots,\OO_X(1,0,\ldots,0)$
  \item $\OO_X(0,\ldots,0,2)$, $\OO_X(0,\ldots,2,0),\ldots,\OO_X(2,0,\ldots,0)$ 
\end{enumerate}
twisted by line bundles of the form $\OO_X(\ell,\ldots,\ell)$ if and only if
\[H^i(E(j_1+t,\ldots,j_s+t)= 0\] for all integers $i,j_1,\cdots,j_s$ and $t$ satisfying 
$1\leq i \leq n_1+\cdots+n_s-1$, $i\leq j_1+\cdots+j_s \leq0$, $-n_i\leq j_i\leq0$ and  for any $j = 1,\ldots,s$,
except for 
\begin{enumerate}
\renewcommand{\theenumi}{\alph{enumi}}
\item $(i,j_1,\ldots,j_s)=(n_i,-n_i,0,\ldots,0)$,$(0,n_i,-n_i,\ldots,0)\ldots(0,\ldots,0,n_i,-n_i)$
\item $(i,j_1,\ldots,j_s)=(n_i,-n_i+1,0,\ldots,0)$,$(0,n_i,-n_i+1,\ldots,0)\ldots(0,\ldots,0,n_i,-n_i+1)$
\end{enumerate}
\end{theorem}

\begin{theorem}
 Let $E$ be a rank $r$ vector bundle on $X=\PP^{n_1}\times\cdots\times\PP^{n_s}$ with $n_1,\dots,n_s\geq1$.
 Let $r_1,\dots,r_s$ be integers such that $0\leq r_i\leq n_i$ for all $i$. Then the vector bundle $E$ is a direct sum of line bundles of
\begin{enumerate}
 \item $\OO_X$
 \item $\OO_X(0,\ldots,0,1),\ldots,\OO_X(0,0,\ldots,r_s)$
 \item $\OO_X(1,\ldots,0),\ldots,\OO_X(r_s,0,\ldots,0)$ 
\end{enumerate}
twisted by line bundles of the form $\OO_X(\ell,\ldots,\ell)$ if and only if
\[H^i(E(j_1+t,\ldots,j_s+t)= 0\] for all integers $i,j_1,\cdots,j_s$ and $t$ satisfying 
$1\leq i \leq n_1+\cdots+n_s-1$, $i\leq j_1+\cdots+j_s \leq0$, $-n_i\leq j_i\leq0$ and  for any $j = 1,\ldots,s$,
except for 
\begin{enumerate}
\renewcommand{\theenumi}{\alph{enumi}}
\item $i=n_1$ and $j_k\geq j_1+n_1-r_1+1$
\item $i=n_k$ and $j_k\leq j_1-n_k+r_k-1$
\end{enumerate}
\end{theorem}

\noindent We also prove that 

\begin{lemma}
 Let $E$ be a vector bundle on $X=\PP^{n}\times\cdots\times\PP^{n}$ that is $\big(\PP^{n}\big)^s$,
Then the following are equivalent:
\begin{enumerate}
\item For any integers $\ell_1,\ldots,\ell_s,\ell$
    \begin{enumerate}
    \renewcommand{\theenumi}{\alph{enumi}}
    \item $H^t(X,E(\ell_1,\cdots,\ell_s)) = 0$ with $|\ell_i-\ell_j|\leq n$ for $i,j=1,\ldots,s$ and $t=1,\dots,n-1,n+1,\ldots,2n-1,2n+1,\ldots,sn-1,sn+1$
    \item $H^n(X,E(\ell,\cdots,\ell)) = 0$
    \end{enumerate}
\item The vector bundle $E$  is isomorphic to a direct sum of line bundles of the form $\mathcal{O}_X(u_1,\ldots,u_s)$ for some
$u_1,\ldots,u_s$ with $|u_i-u_j|\leq n$
\end{enumerate}
\end{lemma}

\section{Preliminaries}

\begin{definition} A coherent sheaf F on $\PP^{n_1}\times\cdots\times\PP^{n_s}$ is said to be $(p_1,\ldots,p_s)$-regular
if, for all $i>0$ whenever $k_1 + \cdots+k_s = -i$ and $-n_j \leq k_j \leq 0$ for any $j = 1,\ldots,s$.
\end{definition}

\begin{definition}
\begin{enumerate}
 \item A coherent sheaf $E$ on $\PP^{n_1}\times\cdots\times\PP^{n_s}$ is said to be 0-regular if, for all $t>0$,
	\[H^t(E(j_1,\ldots,j_s)) = 0\] for all integers $t,j_1,\dots,j_s$ such that $t\geq1$, $j_1 +\ldots + j_s = -t$, $-n_i \leq j_i\leq 0$ for any $i = 1,\ldots,s$.
 \item A coherent sheaf $E$ on $\PP^{n_1}\times\cdots\times\PP^{n_s}$ is said to be $(m_1,\dots,m_s)$-regular if $E(m_1,\cdots,m_s)$ is 0-regular.
 \item Let $E$ be a coherent sheaf on $\PP^{n_1}\times\cdots\times\PP^{n_s}$ and suppose $E$ is 0-regular. For a generic hyperplane $H_i$ of $\mathbb{P}^{n_i}$, $E|_{L_i}$ is 0-regular on $L_i=H_i\times\mathbb{P}^{n_i}$
for all $i=1,\dots,s$.
 \item Let $E$ be a vector bundle on $\PP^{n_1}\times\cdots\times\PP^{n_s}$, suppose that $E$ is 0-regular.
 Then $E(m_1,\ldots,m_s)$ is 0-regular for $m_i\geq0$, $E$ is globally generated.
\end{enumerate}
\end{definition}

\begin{theorem}[K\"{u}nneth formula]
 Let $X$ and $Y$ be projective varieties over a field $k$. 
 Let $\mathscr{F}$ and $\mathscr{G}$ be coherent sheaves on $X$ and $Y$ respectively.
 Let $\mathscr{F}\boxtimes\mathscr{G}$ denote $p_1^*(\mathscr{F})\otimes p_2^*(\mathscr{G})$\\
 then $\displaystyle{H^m(X\times Y,\mathscr{F}\boxtimes\mathscr{G}) \cong \bigoplus_{p+q=m} H^p(X,\mathscr{F})\otimes H^q(Y,\mathscr{G})}$.
\end{theorem}

\begin{lemma}
Let $X=\PP^{n_s}\times\cdots\times\PP^{n_s}$ then\\
$\displaystyle{H^t(X,\OO_X (a_1,\cdots,a_s))\cong \bigoplus_{t=\sum_{i=1}^{s}{t_i}} H^{t_1}(\PP^{n_1},\OO_{\PP^{n_1}}(a_1))\otimes \cdots \otimes H^{t_s}(\PP^{n_s},\OO_{\PP^{n_s}}(a_s))}$
 \end{lemma}

\begin{rem}
\begin{enumerate}
 \item We will often say “regular” instead of “ $(0,\cdots,0)$-regular ”, and “ $p$-regular ” instead of “ $(p,\cdots, p)$-regular ”. 
We define the regularity of $F$, $Reg(F)$, as the least integer $p$ such that $F$ is $p$-regular. We set $Reg(\mathscr{F}) = -\infty$ if there is no such integer.
\item The K\"{u}nneth's formula tells that $\OO(a_1,\cdots,a_s)$ is regular if and only if $a_i\geq0$ , $i=1,\cdots,s$.
In fact \[H^{n_1+\cdots+n_s}(\OO(a_1-n_1,\cdots,a_s-n_s)) = H^{n_1}(\OO(a_1-n_1))\otimes\cdots\otimes H^{n_s}(\OO(a_s-n_s)) = 0\]
if and only if $a_i\geq0$ , $i=1,\cdots,s$.\\
Since
\[H^{n_1}(\OO(a_1-n_1, a_2, \cdots, a_s)) \cong H^{n_1}(\OO(a_1-n_1))\otimes H^0(\OO(a_2))\otimes\cdots\otimes H^0(\OO(a_s))\]
\[H^{n_2}(\OO(a_1, a_2-n_2, \cdots, a_s)) \cong H^0(\OO(a_1))\otimes H^{n_2}(\OO(a_2-n_2))\otimes\cdots\otimes H^0(\OO(a_s))\]
$\cdots\cdots\cdots\cdots\cdots\cdots\cdots\cdots\cdots\cdots\cdots\cdots\cdots\cdots\cdots\cdots\cdots\cdots\cdots$
\[H^{n_s}(\OO(a_1, a_2, \cdots, a_s-n_s)) \cong H^0(\OO(a_1))\otimes H^0(\OO(a_2))\otimes\cdots\otimes H^{n_s}(\OO(a_s-n_s))\]
we see that, if $\OO(a_1,\cdots,a_s)$ is regular, then we have $a_1,\cdots,a_s\geq0$.
Specifically $\OO$ is regular while $\OO(-1,\cdots,-1)$ is not and so $Reg(\OO) = 0$
\end{enumerate}
\end{rem}

\noindent We extend the notion of regularity to $X = \PP^{n_1}\times\cdots\times\PP^{n_s}$ in order to prove some splitting criteria of vector bundles.
We need the following results:

\begin{definition} A vector bundle $E$ on $X$ is arithmetically Cohen-Macaulay (aCM) if, for any $\displaystyle{0<i<\sum_{j=i}^sn_j}$
and for any integer $t$, $H^i(E(t,\cdots,t)) = 0$.
\end{definition}

\begin{proposition}
\begin{enumerate}
 \item The K\"{u}nneth's formula gives: $\OO(a_1,\cdots,a_s)$ is aCM if and only if $a_i-a_j\geq-n_i$ and
$a_j-a_i\geq-n_j$ for $i=1,\cdots,s$.
\item On $X$ we have the following Koszul complexes:
\[K_1:0 \rightarrow \OO(-n_1-1,\cdots,-n_s-1)\rightarrow \OO(-n_1,\cdots,-n_s-1)^{{n_1+1}\choose{n_1}}\rightarrow\cdots\rightarrow
\OO(0,-n_2-1,\cdots-n_s-1)\rightarrow0\]
\[K_2:0 \rightarrow \OO(0,-n_2-1,\cdots-n_s-1)\rightarrow\OO(0,0,\cdots,-n_s-1)\rightarrow \OO(0,0,\cdots,-n_s)^{{n_s+1}\choose{n_s}}\rightarrow\cdots\rightarrow\OO\rightarrow0\]
\[K_3:0 \rightarrow \OO(-n_1-1,\cdots,-n_s-1)\rightarrow \OO(-n_1,\cdots,-n_s-1)^{{n_1+1}\choose{n_1}}\rightarrow\cdots
\rightarrow\OO(0,\cdots-n_s-1)\rightarrow\cdots\rightarrow\OO\rightarrow0\]

\end{enumerate}

\noindent On taking cohomology of the the above exact sequences we get the isomorphisms:
\[H^{n_s}(\OO(0,0,\cdots,-n_s-1))\cong H^{n_1+\cdots+n_s}(\OO(-n_1-1,\cdots,-n_s-1))\]
\[H^{n_j}(\OO(0,\cdots,0,-n_j-1,0,\cdots,0))\cong H^{n_1+\cdots+n_s}(\OO(-n_1-1,\cdots,-n_s-1))\]
\[H^0(\OO)\cong H^{n_s}(\OO(0\cdots,-n_s-1))\]
\[H^0(\OO)\cong H^{n_j}(\OO(0,\cdots,0,-n_j-1,0\cdots,0))\]
\end{proposition}

\noindent The following two results was proved by Miyazaki \cite{15} for a biprojective spaces $\PP^{n_1}\times\PP^{n_2}$
which we extend to a multi projective space $\PP^{n_1}\times\cdots\times\PP^{n_s}$ in the main results section of this paper.

\begin{theorem}[Theorem 1.1 \cite{15}]
Let $E$ be a vector bundle on $\PP^{n_1}_k\times\PP^{n_2}_k$, where $n_1\geq2$ and $n_2\geq2$.
The vector bundle $E$ is a direct sum of line bundles of $\OO_X$, $\OO_X(0,1)$, $\OO_X(0,2)$, $OO_X(1,0)$ and $\OO_X(2,0)$ 
twisted by line bundles of the form $OO_X(\ell,\ell)$ if and only if
\[H^i(E(j_1 + t,j_2 + t)) = 0\]
for all integers $i, j_1, j_2$ and $t$ satisfying that $1\leq i \leq n_1 + n_2 - 1$, $-i\leq j_1 + j_2\leq0$,
$-n_1\leq j_1\leq0$ and $-n_2\leq j_2\leq0$ except for $(i,j_1,j_2)=(n_1, -n_1, 0),(n_1,-n_1 +1, 0),(n_2, 0, -n_2),(n_2, 0, -n_2 + 1)$.
\end{theorem}

\begin{proof}Theorem 1.1\cite{15}\end{proof}

\begin{theorem}[Theorem 2.5 \cite{15}]
Let $E$ be a vector bundle on $\PP^{n_1}_k\times\PP^{n_2}_k$, where $n_1\geq1$ and $n_2\geq1$.
Let $r_1$ and $r_2$ be integers such that $0\leq r_1\leq n_1$ and $0\leq r_2\leq n_2$.
The vector bundle $E$ is a direct sum of line bundles of $\OO_X$, $\OO_X(0,1)\dots\OO_X(0,r_1)$, $\OO_X(1,0)\dots\OO_X(r_2,0)$
twisted by line bundles of the form $OO_X(\ell,\ell)$ if and only if
\[H^i(E(j_1 + t,j_2 + t)) = 0\]
for all integers $i, j_1, j_2$ and $t$ satisfying that $1\leq i \leq n_1 + n_2 - 1$, $-i\leq j_1 + j_2\leq0$,
$-n_1\leq j_1\leq0$ and $-n_2\leq j_2\leq0$ except for either $i=n_1$ and $j_2\geq j_1+n_1-r_1+1$ or $i=n_2$ and $j_2\leq j_1-n_2+r_2-1$
\end{theorem}

\begin{proof}Theorem 2.5\cite{15}\end{proof}

\begin{lemma}[Proposition 5.2 \cite{16}]
 Let $E$ be a vector bundle on $X=\PP^{n}\times\PP^{n}$ then the following are equivalent:
\begin{enumerate}
\item For any integers $\ell_1,\ell_2,\ell$
    \begin{enumerate}
    \renewcommand{\theenumi}{\alph{enumi}}
    \item $H^i(X,E(\ell_1,\ell_2)) = 0$ with $|\ell_1-\ell_2\leq n$ for $i=1,\dots,n-1,n+1,\ldots,2n-1,2n+1$
    \item $H^n(X,E(\ell,\ell)) = 0$
    \end{enumerate}
\item The vector bundle $E$  is isomorphic to a direct sum of line bundles of the form $\mathcal{O}_X(u,v)$ for some
$u,v$ with $|u-v|\leq n$
\end{enumerate}
\end{lemma}

\section{Proof of the Main Results}

\noindent We give the proof of Theorem 1.2.

\begin{proof}
We first prove the "only if" part.\\
Consider the vanishing of the cohomologies
\[H^i(\OO_X(a_1,\ldots,a_s)\otimes\OO_X(j_1+t+a_1,\ldots,j_s+t+a_s))\neq0\] for
\begin{enumerate}
 \item $(0,\dots,0)$
  \item $(0,\ldots,0,1),(0,\ldots,1,0),\ldots,(1,0,\ldots,0)$
  \item $(0,\ldots,0,2),(0,\ldots,2,0),\ldots,(2,0,\ldots,0)$ 
\end{enumerate}
Specifically,
\[H^{n_1}(\OO_X(j_1+t+a_1,\dots,j_s+t+a_s))\neq0\] if and only if $j_1+t+a_1\leq-n_1-1$ and $j_k+t+a_k\geq0$ for $k=2,\ldots,s$.
i.e. if and only if $-j_k-a_k\leq t\leq-n_1-j_1-a_1-1$ for $k=2,\ldots,s$, similarly
\[H^{n_2}(\OO_X(j_1+t+a_1,\dots,j_s+t+a_s))\neq0\] if and only if $j_2+t+a_2\leq-n_2-1$ and $j_k+t+a_k\geq0$ for $k=1,3,4,\ldots,s$.
i.e. if and only if $-j_k-a_k\leq t\leq-n_2-j_2-a_2-1$ for $k=1,3,4,\ldots,s$, 
similarly for all $k=1,2,\ldots,n-1$, 
\[H^{n_s}(\OO_X(j_1+t+a_1,\dots,j_s+t+a_s))\neq0\] if and only if $j_s+t+a_s\leq-n_s-1$ and $j_k+t+a_k\geq0$.
i.e. if and only if $-j_k-a_k\leq t\leq-n_1-j_1-a_1-1$. \\
Thus we get 
$j_k-n_k+a_{k-1}-a_k\leq j_k\leq j_{k-1}+n_{k-1}+a_{k-1}-a_k$ 
\begin{enumerate}
 \item $(0,\dots,0)$
  \item $(0,\ldots,0,1),(0,\ldots,1,0),\ldots,(1,0,\ldots,0)$
  \item $(0,\ldots,0,2),(0,\ldots,2,0),\ldots,(2,0,\ldots,0)$ 
\end{enumerate} all the cohomologies vanish. Thus the ``$\Longleftarrow$" part is true

\vspace{0.5cm}

\noindent We now prove the ``$\Longrightarrow$" part.\\
Let $n=n_+\cdots+n_s$ and suppose that
\[H^n(E(-n_1+t-1,\ldots,-n_s+t-1))\neq0\] then we have that \[H^0(E^{\vee}(-t,\ldots,-t))\neq0\] by Serre's duality and this gives a nonzero map
$E(t,\dots,t)\lra\OO_X$. Thus $(t,\dots,t)-$regularity of $E$ implies that $E(t,\dots,t)$ is globally generated.
Thus we have a nonzero map
\[\oplus\OO_X\lra E(t,\ldots,t)\lra\OO_X\] which is surjective and split and so $\OO_X$ is a direct summand.

Let $n=n_1+\dots+n_s$ and assume that 
\[H^n(E(-n_1+t-1,\ldots,-n_s+t-1))=0\]
we shall focus the case of the nonvanishing of the $n_1^{-th}$ cohomology of $E$ with some twist, for the other cases $n_2,\dots,n_s$ are
proved similarly.
Since $E$ is $(t,\ldots,t)$ regular but not $(t-1,\ldots,t-1)$ regular:\\

\noindent We deal with the follwing cases:
\begin{enumerate}
 \item $H^{n_1}(E(-n_1+t,t-1,\cdots,t-1))\neq0$
 \item $H^{n_1}(E(-n_1+t-1,t-2,\cdots,t-2))\neq0$
 \item $H^{n_1}(E(-n_1+t,t-1,\cdots,t-1))=H^{n_1}(E(-n_1+t,t-2,\cdots,t-2))=0$ but  $H^{n_1}(E(-n_1+t-1,t-1,\cdots,t-1))\neq0$
\end{enumerate}

For 1) there exists $s\in H^{n_1}(-n_1+t,t-1,\cdots,t-1)$. Let $R_a$ be the polynomial ring in $a+1$ variables over $k$ for $a=,1,\cdots,s$.
Take the Koszul complex:
\[K_{a_\bullet}:0\rightarrow F_{a,n_a+1}\rightarrow F_{a,n_a}\rightarrow\cdots\rightarrow  F_{a,r}\rightarrow \cdots\rightarrow F_{a,1}\rightarrow F_{a,0}\rightarrow0\]
where $F_{a,r}$  is a direct sum of ${n_a+1}C_r$ copies of $R_a(-r)$ for $a=1,\cdots,s$.\\
Consider the exact sequence $p_1^*(\overline{K_{1\bullet}})\otimes E(t+1,t-1,\cdots,t-1)$

\[0\lra E(-n_1+t,t-1,\cdots,t-1)\lra E(-n_1+t+1,t-1,\cdots,t-1)^{\oplus{n_1+1}}\]

\[\lra\cdots\lra E(t-r+1,t-1,\cdots,t-1)^{\oplus{n_1+1C_r}}\]

\[\lra\cdots\lra E(t,t-1,\cdots,t-1)^{\oplus{n_1+1}}\lra E(t+1,t-1,\cdots,t-1)\lra0\]

Now for us to construct a surjective morphism

$\phi: H^0(E(t+1,t-1,\cdots,t-1)) \lra H^{n_1}(E(-n_1+t,t-1,\cdots,t-1))$

We need to show that $H^i(E(-n_1+t,t-1,\cdots,t-1)) = 0$ for $i=1,\cdots,n_1$.

We have $H^i(E(t-i+1,t-1,\cdots,t-1)) = H^i(E(-i+1,-1,\cdots,-1)\otimes\OO_X(t,\cdots,t)) = 0$

since $-i\leq(-i+1)+(-1)\leq0$, $-n_k\leq-1\leq0$.

Thus there exists a nonzero element $g\in H^0(E(t+1,t-1,\cdots,t-1))$ such that 

$\phi(g) = s (\neq0) \in H^{n_1}(E(-n_1+t,t-1,\cdots,t-1))$.\\

\noindent Consider the exact sequence
\[p_2^*(\overline{K_{2\bullet}})\otimes E(-t-1,-t+1,\cdots,-t-1)\] that is

\[0\lra E^{\vee}(-t-1,-n_2-t,\cdots,-t-1)\lra E^{\vee}(-t-1,-n_2-t+1,\cdots,-t-1)^{\oplus{n_2+1}}\lra\dots\lra\]

\[E^{\vee}(-t-1,-t-r+1,\cdots,-t-1)^{\oplus{n_2+1C_r}}\lra \cdots\lra E^{\vee}(-t-1,-t,\cdots,-t-1)^{\oplus{n_2+1}}\lra\] \\
\[\lra E^{\vee}(-t-1,-t+1,\cdots,-t-1)\lra0\]

To construct a surjective map

\[\phi:H^0(E^{\vee}(-t-1,-t+1,\cdots,-t-1))\lra H^{n_2}(E^{\vee}(-t-1,-n_2-t,\cdots,-t-1))\] 
we show that  
\[H^i(E^{\vee}(-t-1,-t-i+1,\cdots,-t-1))=0\] for $i=1,\cdots,n_2$ which is equivalent to 
$H^i(E(-n_1+t,n_1+t-i-2,\cdots,-n_1+t))=0$ for $i=n_1,\cdots,n_1+\cdots+n_s-1$ from Serre's duality.\\

\noindent In fact we have
\[H^i(E(-n_1+t,n_1+t-2-i,\dots,-n_1+t)) = \]
\[H^i(E(-n_1+1,n_1-i-1,\dots,-n_1+1)\otimes\OO_X(t-1,t-1,\dots,t-1))=0\]
because  $-i\leq(-n_1+1)+(n_1-i-1)\leq0$, $-n_k\leq n_1-i-1\leq0$.\\

\noindent Now by taking the dual element $s^*\in H^{n_2}(E^{\vee}(-t-1,-n_2-t,\dots,-t-1))$ which corresponds to 

$s^*\in H^{n_1}(E((-n_1+1,t-1,\dots,-t-1))$ we have a nonzero element $f\in H^0(E^{\vee}(-t-1,-t+1,\dots,-t-1))$ such that 
$\psi(f)=s^*(\neq0)\in H^{n_2}(-t-1,-n_2-t,\cdots,-t-1))$.\\
The elements $g$ and $f$ can be viewed as \\
$g\in\Hom(\OO_X(0,\dots,0,2),E(t+1,\dots,t+1))$ and $f\in\Hom(E(t+1,\dots,t+1),\OO_X(0,\dots,0,2))$\\
Now consider the commutative diagram:

\[
\begin{CD}
A@>>>B\\
@VVV@VVV\\
D@>>>C
\end{CD}
\]
where $A=H^0(E(t+1,t-1,\dots,t-1)\otimes H^0(-t-1,-t+1,\dots,-t-1))$,\\
$B=H^0(\OO_X)$,\\
$C={H^{n_1+n_k}(\OO_X(-n_1-1,\dots,,-n_k-1))}$,\\
$D=H^{n_1}(E(t-n_1,t-1,\dots,t-1))\otimes H^{n_k}(E^{\vee}(-t-1,-t-1,\dots,-t-n_k,\dots,-t-1)$
the maps from $A$ to $D$ and $B$ to $C$ give the canonical element. Therefore, $f\circ g$ is an isomorphism and hence 
$\OO_X(0,\cdots,0,2)$ is direct summand of $E(t+1,\dots,t+1)$.\\

\noindent Now for 2), there exists $s(\neq0)\in H^{n_1}(-n_1+t-1,t-2,\dots, t-2)$. Take the corresponding dual maps 
$s^*\in  H^{n_k}(E^{\vee}(-t,\dots,-n_k-t+1,\dots,-t))$ and by Serre duality we have the surjections 
\[\psi:H^0(E(t,\cdots,t-2\dots,t))\]
\[
\begin{CD}
@.@V^{surjects}VV\\
@.H^{n_1}(E(-n_1+t-1,t-2,\dots,t-2))\end{CD}
\]

and

\[\phi:H^0(E^{\vee}(-t,\cdots,-t+2,\dots,-t))\]
\[
\begin{CD}
@.@V^{surjects}VV\\
@.H^{n_k}(E^{\vee}(-t,\cdots,-n_k-t+1,-t,\dots,-t))\end{CD}
\]

\noindent Similarly by Kozsul complex we prove that 
$H^i(E(t-i,t-2,\dots,t-2))=0$ for $1\leq i\leq n_1$ and 

$H^i(E(-n_1+t-1,-n_2+t-1,\dots,n_s+t-i-s+1))=0$ for $1\leq i\leq n_1+\dots+n_s-1$.

We see that 
\[H^i(E(t-i,t-2,\dots,t-2))= H^i(E(-i+1,-1,\cdots,-1)\otimes\OO_X(t-1,\cdots,t-1))=0\] by assumption since
$-i\leq(-i+1)+(-1)\leq 0$, $-n_1\leq-i+1\leq0$ and $-n_k\leq-1\leq0$.

\noindent On the other hand we have 
\[H^i(E(-n_1+t-1,n_1+t-i-3,\dots,-n_1+t-1))=\]

\[H^i(E(-n_1+1,n_1-i-1,\dots,-n_1+1)\otimes\OO_X(t-2,\cdots,t-2))=0\] by assumption since
$-i\leq(-n_1+1)+(n_1-i-1)\leq 0$, $-n_1\leq-n_1+1\leq0$ and $-n_k\leq n_1-i-1\leq0$.\\

\noindent The maps $g\in H^0(E(t,\dots,t,t-2))$ and $f\in H^0(E^{\vee}(-t,\dots,-t,-t+2))$ such that
$\phi(g)=s$ and $\psi(f)=s^*$ give a splitting map $\OO_X(-t,\dots,-t,-t+2)\lra E$ and hence $\OO_X(0,\dots,0,2)$ is a direct summand of
$E(t,\dots,t)$.\\
\noindent Lastly the last case, 3): There exists a non-zero map $s\in H^{n_1}(E(-n_1+t-1,t-1,\dots,t-1))$. To construct

\[\psi:H^0(E(t,t-1\cdots,t-1))\]
\[
\begin{CD}
@.@V^{surjects}VV\\
@.H^{n_1}(E(-n_1+t,t-1\cdots,t-1))\end{CD}
\]
we show that  $H^i(E(t-i,t-1,\dots,t-1))=0$ for $i=1,\dots,n_1$. We have that

\[H^i(E(t-i,t-1,\dots,t-1)) = H^i(E(-i+1,0,\dots,0)\otimes\OO_X(t-1,\dots,t-1))=0\] for $i\neq n_1$
since  $-i\leq -i+1\leq 0$, $-n_1\leq -i+1\leq0$ and $-n_k\leq0$.\\

\noindent Also we have that $H^{n_1}(E(-n_1+t,t-1,\dots,t-1))=0$.

There exists a nonzero element $g\in H^0(t,t-1,\dots,t-1)$ such that $\psi(g)=s(\neq0)\in H^{n_1}(-n_1+t-1,t-1,\dots,t-1))$, likewise
to construct
\[\phi:H^0(E^{\vee}(-t,-t+1\cdots,-t,-t))\]
\[
\begin{CD}
@.@V^{surjects}VV\\
@.H^{n_k}(E(-t,\dots,-t,-t-n_k\cdots,-t))\end{CD}
\]
we show that $H^i(E(t-n_1-1,t+n_1-2-i,\dots,t-n_1-1))=0$ for $i=n_1,\dots,n_!+\dots+n_s-1$. Now we have 

\[H^i(E(-n_1+t-1,n_1+t-2-i,\dots,-n_1+t-1) =\]

\[H^i(E(-n_1+1,n_1-i,\dots,-n_1+1)\otimes\OO_X(t-2,\dots,t-2))=0\]
for $i\neq n_1$ since  $-i\leq(-n_1+1)+(n_1-i)\leq0$, $-n_1\leq-n_1+1\leq0$ and $-n_k\leq n_1-i\leq0$.
Lastly we see that $H^{n_1}(E(t-n_1-1,t-2,\dots,t-2))=0$.\\

\noindent Next, taking the dual element $s^*\in H^{n_k}(E^{\vee}(-t,\dots,-t,-t-n_k,-t,\dots,-t))$ corresponding to
$s\in H^{n_1}(E(t-n_1-1,t-1,\dots,t-1))$,there exists a nonzero element\\
$f\in H^0(E^{\vee}(-t,-t+1,\dots,-t))$ such that
$\phi(f)=s^*(\neq0)\in H^{n_k}(E^{\vee}(-t,\dots,-t,-t-n_k,-t,\dots,-t))$. \\

\noindent The elements $g$ and $f$ can be viewed as elements of  $\Hom(\OO_X(0,\dots,0,1),E(t,\dots,t))$ and $\Hom(E(t,\dots,t),\OO_X(0,\dots,0,1))$
respectively and so $f\circ g$ is an isomorphism hence $\OO_X(0,\dots,0,1)$ is a direct summnad of $E(t,\dots,t)$
\end{proof}

\vspace{1cm}

\noindent We give the proof of Theorem 1.3 which is very similar to the proof of theorem 1.2 above.

\begin{proof}
The proof for the ' 'only if" part is can be shown easily by computing the vanishing of cohomologies.\\

\noindent The proof for the ''if" part is very similar to Theorem 1.2. Take a minimal integer $t$ such that 
$E(t,\dots,t)$ is 0-regular. For the case where 
\[H^n(E(-n_1+t-1,\dots,-n_s+t-1))\neq0\] $\OO_X$ is direct summand of $E(t,\dots,t)$ is in the previous theorem,
particularly, for $r_1=0$ it follows.\\

\noindent Now assuming 
\[H^n(E(-n_1+t-1,\dots,-n_s+t-1))=0\], possible nonvanishing parts for the nonzero regularity of $E(t-1,\dots,t-1)$ appear in
the $n_i$-th cohomologies.

\end{proof}

\noindent \textbf{Data Availability statement}
My manuscript has no associate data.

\noindent \textbf{Conflict of interest}
On behalf of all authors, the corresponding author states that there is no conflict of interest.

\end{document}